\documentclass[11pt,a4paper,reqno]{amsart}

\usepackage{amsmath}
\usepackage{amsfonts}
\usepackage{thmtools}
\usepackage{graphicx}
\usepackage{color}
\usepackage{enumerate}
\usepackage{mathrsfs}
\usepackage{ifthen}
\usepackage{caption}

\newcommand{\R}{\mathbb{R}}
\newcommand{\C}{\mathbb{C}}

\newcommand{\F}{\mathscr{F}}
\renewcommand{\S}{\mathscr{S}}
\newcommand{\D}{\mathscr{D}}

\newcommand{\bdry}{\partial}
\renewcommand{\O}{\mathcal{O}}
\renewcommand{\phi}{\varphi}
\renewcommand{\epsilon}{\varepsilon}
\newcommand{\M}{\ensuremath{\mathcal{M}}}

\newcommand{\bal}{\operatorname{Bal}}
\newcommand{\supp}{\operatorname{supp}}
\newcommand{\diam}{\operatorname{diam}}

\newcommand{\logcap}{\operatorname{cap}}
\newcommand{\ie}{\emph{i.e.} }
\newcommand{\cf}{\textrm{cf.\ }}
\newcommand{\eg}{\textrm{e.g.\ }}

\renewcommand{\d}{\ensuremath{\mathrm{d}}}
\newcommand{\dd}{\ensuremath{\,\mathrm{d}}}

\newcommand{\disk}[2]{D\left(#1,#2\right)}
\newcommand{\restr}[2]{{#1}|_{#2}}
\newcommand{\restrx}[2]{{#1}\lfloor{#2}}

\theoremstyle{plain}
\newtheorem{theorem}{Theorem}[section]

\newtheorem{proposition}[theorem]{Proposition}
\newtheorem{corollary}[theorem]{Corollary}
\theoremstyle{definition}
\newtheorem{definition}[theorem]{Definition}
\newtheorem{remark}[theorem]{Remark}

\title{Equilibrium Measures and Partial Balayage}
\subjclass[2010]{Primary 31A05, Secondary 35R35.}
\date{September 20, 2013}
\keywords{Equilibrium problem, partial balayage, obstacle problem.}

\author{Joakim Roos}

\begin{document}

\begin{abstract}
  We consider the equilibrium problem for an external background
  potential in weighted potential theory, and show that for a large
  class of background potentials there is a complementarity
  relationship between the measure solving the weighted equilibrium
  problem---the weighted equilibrium measure---and a certain partial
  balayage measure.
\end{abstract}
\maketitle

\section{Introduction}
\label{sec:introduction}
Solving the equilibrium problem for a given external background
potential, tantamount to finding the so-called \emph{equilibrium
  measure} of the background potential, is a problem that in recent
time has turned out to have various connections to many different
fields, two of which are locating the zeros of
orthogonal polynomials and finding the eigenvalues of random normal
matrices; we refer to
\cite{WiegmannZabrodin:03,TeodorescuEtAl:05,Teodorescu:06,BaloghHarnad:09,
  HedenmalmMakarov:13} for more information. For sake of studying the
type of equilibrium measures that can arise in these problems, Balogh
and Harnad introduced in \cite{BaloghHarnad:09} the class of
superharmonically perturbed Gaussian background potentials of the form
\begin{gather*}
  Q(z) = \alpha |z|^2 + U^\nu(z),
\end{gather*}
where $\alpha > 0$ and is $\nu$ is a finite positive Borel measure
with compact support. Based on some of their results it is clear that
there is a strong connection between the weighted equilibrium measure
arising from such a background potential and a certain \emph{partial
  balayage} operation, a connection which has, to the best of our
knowledge, not yet been treated in detail in the literature. It turns
out that the measure resulting from the partial balayage operation and
the weighted equilibrium measure are in a sense complementary to each
other, as is shown below in Theorem~\ref{thm:wpt-pb-rel}.

This paper is organized as follows: in section~\ref{sec:wpt-and-em} we
review the fundamental notions and results from weighted potential
theory that will be used in the remainder of the article, in
section~\ref{sec:partial-balayage} we define partial balayage in full
generality and state some of the important properties, such as
translational invariance, and in section~\ref{sec:examples-part-i} we
discuss one sample background potential treated in detail in
\cite{BaloghHarnad:09} that inspired the current work, and give a
brief explanation as to from where the idea of the complementarity
relationship between partial balayage and the weighted equilibrium
measure came. The subsequent two sections are then devoted to proving
this complementarity relationship in a bit more general setting:
section~\ref{sec:props-cert-pb-meas} treats existence and certain
properties of the particular partial balayage operation that we shall
employ, while section~\ref{sec:equil-meas-thr-pb} combines the results
from weighted potential theory with those for the partial balayage
operation discussed in section~\ref{sec:props-cert-pb-meas} to prove
the stated complementarity relationship,
Theorem~\ref{thm:wpt-pb-rel}. In the final part of the paper,
section~\ref{sec:examples-part-ii}, we treat the sample case of
superharmonically perturbed Gaussian background potentials in full
detail as an illustration of how Theorem~\ref{thm:wpt-pb-rel} can be
employed.

We will in this paper use the following notation:

\begin{center}
  \begin{tabular}{l l}
    $\hat{\C}$ & Riemann sphere \\
    $\disk{a}{r}$ & Open disk centered at $a \in \C$ with radius $r$
    \\
    $\diam(S)$ & Diameter of the set $S$; $\diam(S) = \sup_{x,y \in S}
    |x - y|$ \\
    $m, s$ & (Planar) Lebesgue measure, arc length measure \\
    $\delta_p$ & Dirac measure at a point $p \in \C$ \\
    $\restrx{\mu}{S}$ & Restriction of the (possibly signed) measure
    $\mu$ to \\
    & the set $S$, extended by zero outside $S$ \\
    $\restr{f}{S}$ & Restriction of the function $f$ to the set $S$ \\
    $U^\mu$ & Logarithmic potential of a (signed Radon) measure $\mu$; \\
    & $U^{\mu}(z) = \int \log |z - \zeta|^{-1} \dd \mu(\zeta)$ \\
    $U^S$ & Logarithmic potential of the Lebesgue measure restricted \\
    & to the set $S$; $U^S \equiv U^{\restrx{m}{S}}$\\
    $\Delta$ & Laplace operator; $\Delta = \partial_x^2
    + \partial_y^2$ with $z = x + i y \in \C$\\
    $I(\mu)$ & Logarithmic energy of $\mu$; \\
    & $I(\mu) = \iint \log |z - \zeta|^{-1} \dd \mu(\zeta) \d \mu(z) =
    \int U^\mu \dd \mu$\\
    $\logcap(S)$ & Logarithmic capacity of the set $S$ \\
    $Q$ & Background potential/external field \\
    $\M_t(E)$ & Set of Borel measures $\mu$ such that $\mu(\C) = t$\\
    & and $\supp \mu \subseteq E$ \\
    $I_{Q,t}(\mu)$ & Weighted energy of a measure $\mu \in \M_t(E)$ \\
    $\mu_{Q,t}$ & Equilibrium measure with total mass $t$ of \\
    & a ($t$-admissible) background potential $Q$ \\
    $F_{Q,t}$ & Modified Robin constant corresponding to $\mu_{Q,t}$ \\
    $\S$ & Set of subharmonic functions in $\C$ \\
    $\S_t$ & Set of functions $V(z) \in \S$ that are harmonic near
    infinity\\
    &  and $V(z) \leq t \log |z| + \O(1)$ for large $|z|$ (with $t >
    0$) \\
    $\bal(\mu, \lambda)$ & Partial balayage of $\mu$ to $\lambda$
  \end{tabular}
\end{center}
Whenever we refer to a \emph{measure} in this article, we always mean
a positive measure. Every measure considered here will at the very
least be a Borel measure, and most of them will be Radon measures
(where, in our setting, Radon simply means that the measure is finite
on every compact set; note that a Radon measure with compact support
by necessity is a finite measure).  We shall often utilize signed
Radon measures, which simply are differences of two positive Radon
measures. For any such signed Radon measure $\sigma$ we can decompose
it as $\sigma = \sigma_+ - \sigma_-$, where we take $(\sigma_+,
\sigma_-)$ to be the Jordan decomposition of $\sigma$, \ie the minimal
decomposition into positive measures such that $\sigma_+ \perp
\sigma_-$. Relations between objects of not necessarily the same type
(for instance an inequality between a function and a measure) will
implicitly always be in the sense of distributions; the same applies
to the Laplace operator, which, when applied to functions not
necessarily twice differentiable, also always will be taken in the
sense of distributions.

\section{Weighted Potential Theory and Equilibrium Measures}
\label{sec:wpt-and-em}

We closely follow \cite{SaffTotik:97}. Let $E \subseteq \C$ be a
subset of the plane; throughout we always assume that $E$ is a
\emph{closed} set. We denote by $\M_t(E)$ the set of (positive) Borel
measures with support in $E$ and total mass $t > 0$; as a minor remark
note that every measure in $\M_t(E)$ by necessity is a Radon measure.

\begin{remark}
  \label{rem:wpt-prob-t-rel}
  The weighted potential theory as described in \cite{SaffTotik:97} is
  stated in terms of probability measures, \ie measures with total
  mass one. As we soon shall see, it will for our purposes be more
  convenient to allow these measures to have an arbitrary (but finite)
  total mass $t$. One can easily translate between the two
  formulations of the theory, since if we let $\mu_{Q,t}$ be the
  equilibrium measure of total mass $t$ for a $t$-admissible
  background potential $Q$ (as defined in
  Definition~\ref{def:t-adm-bkg-pot} and
  Proposition~\ref{prop:def-eq-meas}) and $\hat{\mu}_{\hat{Q}}$ be the
  equilibrium measure of total mass one for the background potential
  $\hat{Q} := Q / t$ (as defined in \cite{SaffTotik:97}), then
  $\mu_{Q,t} = t \hat{\mu}_{\hat{Q}}$.
\end{remark}

\begin{definition}
  \label{def:t-adm-bkg-pot}
  A function $Q : E \to (-\infty, \infty]$ is called a
  \emph{$t$-admissible background potential (on $E$)} for $t > 0$ if
  the following holds:
  \begin{enumerate}[(i)]
  \item $Q$ is lower semicontinuous,
  \item $\logcap(\{ z \in E : Q(z) < \infty \}) > 0$, and
  \item $Q(z) - t \log |z| \to \infty$ as $|z| \to \infty$, $z \in E$
    (if $E$ is unbounded).
  \end{enumerate}
\end{definition}

\begin{remark}
  If the set $E$ in the above definition is bounded (\ie if $E$ is
  compact) we see that requirement (iii) is empty and that the
  parameter $t$ has no influence whatsoever on the admissibility of a
  background potential $Q$. Whenever the domain of $Q$ is compact we
  therefore for ease of notation simply say that $Q$ is an
  \emph{admissible background potential} if it satisfies properties
  (i) and (ii) in Definition~\ref{def:t-adm-bkg-pot}. Moreover, if $Q$
  is a $t$-admissible background potential for every $t > 0$ we say
  that $Q$ is a \emph{fully admissible background potential}.
\end{remark}

\begin{definition}
  Given a $t$-admissible background potential $Q$ on $E$ we let $w(z)
  := e^{-Q(z)}$ and define the \emph{weighted energy} $I_{Q,t}(\mu)$
  of $\mu \in \M_t(E)$ by
  \begin{align*}
    I_{Q,t}(\mu) :=& \frac{1}{t} \iint \log \frac{1}{|z - \zeta|^t w(z)
      w(\zeta)} \dd \mu(\zeta) \dd \mu(z) \\
    =& \int U^\mu \dd \mu + 2 \int Q \dd \mu.
  \end{align*}
\end{definition}

The following existence property, essentially due to Frostman
\cite{Frostman:35}, for a solution to the equilibrium problem of an
external background potential is an important part of weighted
potential theory (see \eg \cite{Ransford:95} or
\cite[Theorem~I.1.3]{SaffTotik:97} for proofs of the case $t=1$, the
relation described in Remark~\ref{rem:wpt-prob-t-rel} easily yields
the proposition for arbitrary values of $t$):

\begin{proposition}
  \label{prop:def-eq-meas}
  Let $Q$ be a $t$-admissible background potential on $E$ for some $t
  > 0$, and define
  \begin{gather*}
    V_{Q,t} := \inf_{\mu \in \M_t(E)} I_{Q,t}(\mu).
  \end{gather*}
  Then $V_{Q,t}$ is finite, and there is a unique measure $\mu_{Q,t}
  \in \M_{t}(E)$ such that $I_{Q,t} (\mu_{Q,t}) = V_{Q,t}$. The
  measure $\mu_{Q,t}$ has compact support, and is called the
  \emph{equilibrium measure} of the background potential $Q$.
\end{proposition}

Determining the equilibrium measure $\mu_{Q,t}$ is in applications
often an important problem but also one that might be difficult to
solve; the following result is quite useful:

\begin{proposition}
  \label{prop:char-of-eq-meas}
  Let $Q$ be a $t$-admissible background potential on $E$ for some $t
  > 0$, and assume that $\mu \in \M_t(E)$ has finite logarithmic
  energy (\ie $|I(\mu)| < \infty$) and that $\supp \mu$ is compact. If
  there is a constant $F$ such that
  \begin{align}
    U^\mu + Q &\geq F \qquad \text{q.e. on }
    E, \label{eq:req-1-for-eq-meas} \\
    U^\mu + Q &= F \qquad \text{q.e. on }
    \supp \label{eq:req-2-for-eq-meas} \mu,
  \end{align}
  then $\mu = \mu_{Q,t}$ and $F = F_{Q,t}$, where $F_{Q,t}$ is the
  \emph{modified Robin constant}
  \begin{gather*}
    F_{Q,t} := \frac{1}{t} \left( V_{Q,t} - \int Q \dd \mu_{Q,t}
    \right).
  \end{gather*}
\end{proposition}

We refer to \cite[Theorem~I.3.3]{SaffTotik:97} for a proof of the
above. For future use we note the following: if we are given a
$t$-admissible background potential $Q$ on some set $E$ and for some
reason have prior knowledge that the support of the equilibrium
measure $\mu_{Q,t}$ is fully contained in some subset $\hat{E} \subset
E$, then we can just as well switch to studying the equilibrium
problem for the restriction of $Q$ to the set $\hat{E}$ instead; it is
readily verified that $\mu_{\restr{Q}{\hat{E}},t} = \mu_{Q,t}$.

In section~\ref{sec:equil-meas-thr-pb} we are going to utilize the
formulation of the problem of determining the potential of the
equilibrium measure as a classical obstacle problem, for the sake of
showing a relationship between the equilibrium measure and a partial
balayage measure. For this purpose the following result, a slight
reformulation of Theorem~I.4.1 in \cite{SaffTotik:97} better suited
for our purposes, will be useful:

\begin{proposition}
  \label{prop:wpt-upper-envelope}
  Let $t > 0$ be fixed, and define $\S_t$ to be the set of all
  subharmonic functions $V(z)$ on $\C$ that for large $|z|$ are both
  harmonic and such that $V(z) \leq t \log |z| + \O(1)$, \ie $V(z) - t
  \log |z|$ is bounded from above near $\infty$. Then, for a
  $t$-admissible background potential $Q$ on a set $E$, the function
  \begin{gather*}
    F_{Q,t} - U^{\mu_{Q,t}}(z)
  \end{gather*}
  is the upper envelope of the functions $V$ in $\S_t$ satisfying
  $V(z) \leq Q(z)$ for quasi-every $z \in E$.
\end{proposition}

We are also going to need the so called \emph{Principle of Domination}
(we refer to Theorem~II.3.2 in \cite{SaffTotik:97} for a proof):

\begin{proposition}
  \label{prop:princ-of-domination}
  Let $\mu$ and $\nu$ be two positive finite Borel measures with
  compact support on $\C$, and suppose that the total mass of $\nu$
  does not exceed that of $\mu$. Assume further that $\mu$ has finite
  logarithmic energy. If, for some constant $c$, the inequality
  \begin{gather*}
    U^\mu(z) \leq U^\nu(z) + c
  \end{gather*}
  holds $\mu$-almost everywhere, then it holds for all $z \in \C$.
\end{proposition}

\section{Partial Balayage}
\label{sec:partial-balayage}

The concept of \emph{balayage} is well-known in classical potential
theory, and is essentially the process of clearing the density of a
measure in a prescribed region and altering the measure on the
boundary of the region in such a way that the potential of the measure
remains unchanged in the exterior of the region in
question. \emph{Partial balayage} is a related concept, in the sense
that we allow some density to remain after the process, while still
demanding that the potential is unchanged in some region. Partial
balayage can be defined in a multitude of (albeit often equivalent)
ways \cite{Gustafsson:04}, and we will in this paper use the following
rather broad definition (see also \cite{GustafssonSakai:94, Sjodin:07,
  GardinerSjodin:09, GardinerSjodin:08}):

\begin{definition}
  \label{def:partial-balayage}
  Let $\mu$ and $\lambda$ be (signed) Radon measures with compact
  supports. Set
  \begin{gather*}
    \F^\mu_{\lambda} := \{ V \in \D'(\C) : V \leq U^\mu \textrm{ in }
    \C, - \frac{1}{2\pi} \Delta V \leq \lambda \textrm{ in } \C \},
  \end{gather*}
  where $\D'(\C)$ denotes the set of distributions on $\C$. The set
  $\F^\mu_\lambda$ is of course highly dependent on the nature of
  $\mu$ and $\lambda$, but if it is nonempty and contains a largest
  element $V^\mu \equiv V^\mu_\lambda := \sup \F^\mu_\lambda$ then we
  define the \emph{partial balayage} of $\mu$ to $\lambda$ to be the
  signed Radon measure
  \begin{gather*}
    \bal(\mu, \lambda) := - \frac{1}{2 \pi} \Delta V^\mu.
  \end{gather*}
\end{definition}

\begin{remark}
  \label{rem:pb-transl-inv}
  Let us just note a sometimes very useful sort of translation
  invariance property of the partial balayage measure. As a
  consequence of the definition it is immediate that if $\mu, \lambda$
  and $\sigma$ are suitable measures for which either $\F^\mu_\lambda$
  or $\F^{\mu + \sigma}_{\lambda + \sigma}$ is nonempty then by
  necessity $\F^{\mu + \sigma}_{\lambda + \sigma} = \F^\mu_\lambda +
  U^\sigma$, and so it follows directly that
  \begin{gather}
    \label{eq:pb-transl-inv-prop} \bal(\mu + \sigma, \lambda + \sigma)
= \bal(\mu, \lambda) + \sigma.
  \end{gather}
  Existence of $\bal(\mu, \lambda)$ hence implies, through
  \eqref{eq:pb-transl-inv-prop}, the existence of any partial balayage
  measure of the form $\bal(\mu + \sigma, \lambda + \sigma)$, and vice
  versa. 
\end{remark}

In this paper we will exclusively work with partial balayage to the
zero measure, \ie we will use partial balayage measures of the form
$\bal(\sigma, 0)$; although somewhat similar, note that such a partial
balayage measure in general will \emph{not} be the same as a classical
balayage measure. Under certain assumptions on the signed Radon
measure $\sigma$ partial balayage measures of this form are quite
well-behaved, as we shall see in
section~\ref{sec:props-cert-pb-meas}.

\section{Examples: Part I}
\label{sec:examples-part-i}

Having discussed the basics of weighted potential theory and partial
balayage let us consider an example before we move on to the main
results. As mentioned in the introduction, the inspiration for this
paper comes from \cite{BaloghHarnad:09}, which studies a class of
superharmonically perturbed Gaussian background potentials of the form
\begin{gather}
  \label{eq:ex-p1-superharm-bkg-pots}
  Q(z) = Q_{\alpha, \nu}(z) = \alpha |z|^2 + U^\nu(z),
\end{gather}
defined on $\C$, where $\alpha > 0$ is a parameter and $\nu$ is a
compactly supported finite positive Borel measure. Any background
potential of this form is readily seen to be fully admissible
(\cite[Proposition~2.4]{BaloghHarnad:09} shows this when $t=1$, and
the argument is easily generalized to arbitrary $t$ using the relation
described in Remark~\ref{rem:wpt-prob-t-rel}). The perhaps simplest
non-trivial background potential of the form
\eqref{eq:ex-p1-superharm-bkg-pots} is obtained by taking $\nu = \beta
\delta_a$, where $\beta > 0$ and $a \in \C$, \ie $Q$ becomes
\begin{gather}
  \label{eq:sample-bkg-pot-form}
  Q(z) = \alpha |z|^2 + \beta \log \frac{1}{|z - a|}.
\end{gather}
This background potential is treated in detail in
\cite{BaloghHarnad:09}, and its equilibrium measure (for $t=1$) is in
\cite[Proposition~3.3]{BaloghHarnad:09} determined completely by the
following (\cf Figure~\ref{fig:ex-p1-fig}):

\begin{proposition}
  \label{prop:ex-p1-bh-prop}
  Let the two radii $R$ and $r$ be defined by
  \begin{gather*}
    R := \sqrt{\frac{1+\beta}{2 \alpha}} \quad \text{and} \quad r :=
    \sqrt{\frac{\beta}{2 \alpha}}.
  \end{gather*}
  The equilibrium measure $\mu_{Q}\ ( = \mu_{Q,1})$ of the background
  potential \eqref{eq:sample-bkg-pot-form} is absolutely continuous
  with respect to the Lebesgue measure with density $2\alpha / \pi$,
  \ie
  \begin{gather*}
    \mu_Q = \frac{2\alpha}{\pi} \restrx{m}{S_Q},
  \end{gather*}
  where the support $S_Q = \supp \mu_Q$ depends on the geometric
  arrangement of the disks $\disk{a}{r}$ and $\disk{0}{R}$ in the
  following way:
  \begin{enumerate}[(i)]
  \item If $\disk{a}{r} \subset \disk{0}{R}$ then
    \begin{gather*}
      S_Q = \overline{\disk{0}{R}} \setminus \disk{a}{r}.
    \end{gather*}
  \item If $\disk{a}{r} \not\subset \disk{0}{R}$ then $\hat{\C}
    \setminus S_Q$ is given by a rational exterior conformal mapping
    of the form
    \begin{gather*}
      f: \hat{\C} \setminus \{\zeta : |\zeta| \leq 1\} \to \hat{\C}
      \setminus S_Q, \quad f(\zeta) = \rho \zeta + u + \frac{v}{\zeta
        - A},
    \end{gather*}
    where the coefficients $\rho \in \R^+$, $0 < |A| < 1$ and $u, v
    \in \C$ of the mapping $f(\zeta)$ are uniquely determined by the
    parameters $\alpha, \beta$ and $a$ of the potential $Q(z)$.
  \end{enumerate}
\end{proposition}

\begin{center}
  \begin{figure} \centering
    \begin{minipage}[t]{0.35\textwidth} \centering
      \includegraphics[scale=0.6]{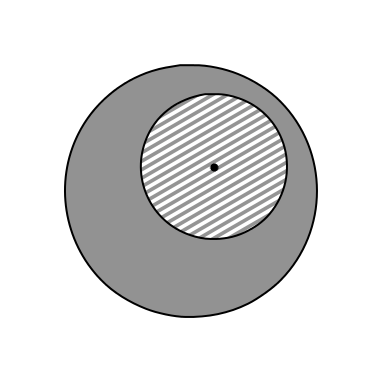}
    \end{minipage}
    \begin{minipage}[t]{0.53\textwidth} \centering
      \includegraphics[scale=0.6]{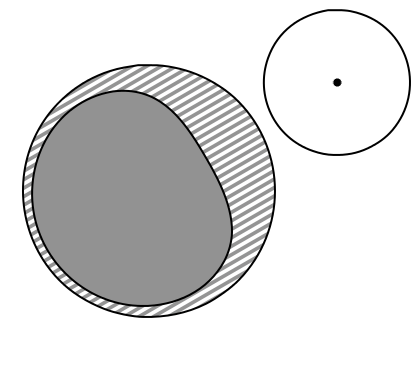}
    \end{minipage}
    \caption{Illustration of the supports of $\mu_Q$ (gray) and
      $\frac{2\alpha}{\pi} \restrx{m}{\overline{\disk{0}{R}}\setminus
        S_Q}$ (stripes) in the cases of the point charge $\beta
      \delta_a$ (black dot) being located inside the disk
      $\overline{\disk{0}{R}}$ (case (i) in
      Proposition~\ref{prop:ex-p1-bh-prop}, on the left) and outside
      $\overline{\disk{0}{R}}$ (case (ii), on the right). (Adaptation
      of Figure~3.1 in \cite{BaloghHarnad:09})}
    \label{fig:ex-p1-fig}
  \end{figure}
\end{center}

One interpretation of the above result that becomes particularly clear
when we look at Figure~\ref{fig:ex-p1-fig} is that the equilibrium
measure $\mu_Q$ effectively is the measure that is ``removed'' from
the measure $\frac{2\alpha}{\pi} \restrx{m}{\overline{\disk{0}{R}}}$
when we try to ``sweep'' the point charge $\beta$ at $a \in \C$ onto
the disk $\overline{\disk{0}{R}}$ (while factoring in the constant
density $2\alpha / \pi$). In terms of the partial balayage defined in
the previous section we can interpret this as
\begin{align*}
  \bal\left(\beta \delta_a, \frac{2\alpha}{\pi}
    \restrx{m}{\overline{\disk{0}{R}}}\right) &= \frac{2\alpha}{\pi}
  \restrx{m}{\overline{\disk{0}{R}} \setminus S_Q} \\
  &= \frac{2\alpha}{\pi} \left( \restrx{m}{\overline{\disk{0}{R}}} -
    \restrx{m}{S_Q} \right)
\end{align*}

As noted in Remark~\ref{rem:pb-transl-inv} we can rewrite this
as
\begin{gather*}
  \bal\left( \beta \delta_a - \frac{2\alpha}{\pi}
    \restrx{m}{\overline{\disk{0}{R}}}, 0 \right) = -
  \frac{2\alpha}{\pi} \restrx{m}{S_Q} = -\mu_Q,
\end{gather*}
or, equivalently,
\begin{gather}
  \label{eq:ex-p1-main-rel}
  \mu_Q + \bal(\sigma, 0) = 0
\end{gather}
where
\begin{gather*}
  \sigma = \beta \delta_a - \frac{2\alpha}{\pi}
  \restrx{m}{\overline{\disk{0}{R}}}. 
\end{gather*}
Of course, to fully justify the above we need to be careful and in
detail verify each step. We omit the precise calculations, and in the
following two sections instead formulate a more general theory of when
relation \eqref{eq:ex-p1-main-rel} holds. As an illustration of how to
use the results we obtain we will then in
section~\ref{sec:examples-part-ii} apply the theory to the entire
class of background potentials of the form
\eqref{eq:ex-p1-superharm-bkg-pots}, hence in particular also to the
background potential \eqref{eq:sample-bkg-pot-form} used above.

\section{Properties of certain Partial Balayage Measures}
\label{sec:props-cert-pb-meas}

As previously mentioned we will exclusively consider partial balayage
of a signed Radon measure to the zero measure, in particular partial
balayage of the form $\bal(\sigma, 0)$ for some signed Radon measure
$\sigma$. For later use we need to determine some of the properties of
this type of partial balayage measure, given certain assumptions on
the signed measure on which we perform the balayage operation. We
require the following:

\begin{theorem}
  \label{thm:pb-properties-with-assumptions}
  Let $\sigma = \sigma_+ - \sigma_-$ be a signed Radon measure with
  compact support for which $\sigma(\C) < 0$ and such that
  $U^{\sigma_-}$ is continuous on $\C$. Then $\bal(\sigma, 0)$ exists,
  has the same total mass as $\sigma$, satisfies
  \begin{gather}
    \label{eq:pb-supp-in-sigma-neg}
    \supp \bal(\sigma, 0) \subseteq \supp \sigma_- ,
  \end{gather}
  and has finite logarithmic energy. Moreover, with $V^\sigma$ as in
  Definition~\ref{def:partial-balayage} for $\nu := \bal(\sigma, 0)$,
  define the sets $\omega := \{z \in \C : V^\sigma(z) < U^\sigma(z)
  \}$ and $\Omega := \C \setminus \supp \nu$. Then $\omega$ is an open
  set and $\omega \subseteq \Omega$, so for every $z \in \supp \nu$ we
  have $V^\sigma(z) = U^\sigma(z)$. Furthermore, there exists a
  constant $c_0$ such that $V^\sigma = U^{\nu} + c_0$.
\end{theorem}

\begin{proof}
  Let us first show existence of the partial balayage measure. Since
  $U^{\sigma} = U^{\sigma_+} - U^{\sigma_-}$ and the potential of a
  positive measure with our definition is superharmonic, hence lower
  semicontinuous, it clearly follows from the assumption of continuity
  of $U^{\sigma_-}$ that the potential $U^\sigma$ is lower
  semicontinuous on the whole of $\C$. We assume that $\sigma$ has
  compact support, but since lower semicontinuous functions are
  bounded from below on any compact set we can conclude, if we also
  use the fact that as $|z| \to \infty$ we get
  \begin{gather*}
    U^\sigma(z) = \sigma(\C) \log \frac{1}{|z|} + \O \left(
      \frac{1}{|z|} \right) = \underbrace{(-\sigma(\C))}_{>0} \log |z|
    + \O \left( \frac{1}{|z|} \right) \to \infty,
  \end{gather*}
  that there exists $M \in \R$ such that $M \leq U^\sigma(z)$ for all
  $z \in \C$. If we now let $V(z) := M$ on $\C$ we see that both $V
  \leq U^\sigma$ and $\Delta V \geq 0$ hold on the whole of $\C$
  (since $\Delta V = 0$). The function $V$ is thus a competing element
  for the function $V^\sigma \equiv V^\sigma_0 := \sup \F^\sigma_0$ as
  in Definition~\ref{def:partial-balayage}, \ie the set $\F^\sigma_0$
  is nonempty. If we now can show that $V^\sigma \in \F^\sigma_0$ then
  the existence of $\bal(\sigma, 0)$ follows. That $V^\sigma$ indeed
  belongs to $\F^\sigma_0$ is nontrivial, but can be shown using the
  same argument as in Lemma~1 in \cite{GardinerSjodin:09}: under the
  given assumptions it is relatively easily seen that $\F^\sigma_0$ is
  locally uniformly bounded above, so the upper semicontinuous
  regularization $V^\sigma_r$ of $V^\sigma$ satisfies $\Delta
  V^\sigma_r \geq 0$ and is equal to $V^\sigma$ almost everywhere
  (with respect to the Lebesgue measure), hence everywhere by taking
  means over arbitrary small balls. Since $V^\sigma = \sup
  \F^\sigma_0$ clearly satisfies $V^\sigma \leq U^\sigma$ it now
  follows that $V^\sigma = \sup \F^\sigma_0 = V^\sigma_r \in
  \F^\sigma_0$, and we can conclude that the partial balayage measure
  $\bal(\sigma, 0)$ indeed exists.

  The proposition for the location of the support of the partial
  balayage measure under the given assumptions on $\sigma$, \ie
  property \eqref{eq:pb-supp-in-sigma-neg}, is clearly a corollary
  from the fact that we have the inequalities
  \begin{gather}
    \label{eq:pb-ineqs}
    -\sigma_- \leq \bal(\sigma, 0) \leq 0.
  \end{gather}
  It is clear by the definition of $\bal(\sigma, 0)$ that
  $\bal(\sigma, 0) \leq 0$ holds everywhere. The leftmost inequality
  in \eqref{eq:pb-ineqs} follows from the fact that we can use
  \eqref{eq:pb-transl-inv-prop} to write
  \begin{gather*}
    \bal(\sigma, 0) = \bal(\sigma_+ - \sigma_-, \sigma_ - - \sigma_-)
    = \bal(\sigma_+, \sigma_-) - \sigma_-.
  \end{gather*}
  Although \cite{GustafssonSakai:94} uses a slightly less general
  definition of partial balayage than the one used in this paper,
  applying the same argument as that used in the proof of part (b) of
  Theorem~2.1 in \cite{GustafssonSakai:94} shows that $\bal(\sigma_+,
  \sigma_-) \geq 0$ holds under the given assumptions, and
  \eqref{eq:pb-ineqs} follows.

  As for $\bal(\sigma, 0)$ having finite logarithmic energy, we start
  by making the observation that if $d := \diam(\supp \sigma_-)$, then
  $d > 0$. Indeed, assume this is not the case: then either $\supp
  \sigma_- = \emptyset$ or $\supp \sigma_- = \{a\}$ holds, for some $a
  \in \C$. The first case implies $\sigma_- \equiv 0$, which
  contradicts $\sigma(\C) < 0$, and the second case implies that
  $\sigma_-$ is a Dirac point mass, which contradicts the assumed
  continuity of $U^{\sigma_-}$ on the whole of $\C$. It follows that,
  indeed, $d > 0$.

  Let, for sake of simplicity, $\mu := - \bal(\sigma, 0)$; evidently
  $\bal(\sigma, 0)$ has finite logarithmic energy if and only if $\mu$
  has finite logarithmic energy. We now look at $I(\mu)$ in detail. By
  definition we get
  \begin{align}
    I(\mu) &= \iint \log \frac{1}{|z - \zeta|} \dd \mu(\zeta) \d
    \mu(z) \nonumber \\
    &= \iint \log \left| \frac{d}{z - \zeta} \right| \dd \mu(\zeta) \d
    \mu(z) - \mu(\C)^2 \log d. \label{eq:l-pbfinen-1}
  \end{align}
  For any $z, \zeta \in \supp \mu$ it follows from
  \eqref{eq:pb-supp-in-sigma-neg} that $|z - \zeta| \leq d$, so in
  particular we have $\log \left| d / (z - \zeta) \right| \geq 0$ for
  all $z, \zeta \in \supp \mu$; note that this immediately implies
  $I(\mu) > -\infty$. We see that \eqref{eq:pb-ineqs} yields $\mu \leq
  \sigma_-$, so \eqref{eq:l-pbfinen-1} implies that
  \begin{align*}
    I(\mu) &\leq \iint \log \left| \frac{d}{z-\zeta} \right| \dd
    \sigma_-(\zeta) \d \sigma_-(z) - \mu(\C)^2 \log d \\
    &= \iint \log \frac{1}{|z-\zeta|} \dd \sigma_-(\zeta) \d
    \sigma_-(z) +(\sigma_-(\C)^2- \mu(\C)^2) \log d \\
    &= I(\sigma_-) + (\sigma_-(\C)^2 - \mu(\C)^2) \log d.
  \end{align*}
  The proposed finite logarithmic energy now clearly follows if we can
  show that $I(\sigma_-)$ is finite. However, this is nearly trivial:
  we assume that $U^{\sigma_-}$ is continuous on $\C$ and that $\supp
  \sigma_-$ is compact, from which we obtain
  \begin{gather*}
    |I(\sigma_-)| = \left| \int U^{\sigma_-} \dd \sigma_- \right| \leq
    \max_{z \in \supp \sigma_-} |U^{\sigma_-}(z)| \cdot \sigma_-(\C) <
    \infty.
  \end{gather*}

  Having established the existence of the partial balayage measure,
  the location of its support and it having finite logarithmic energy,
  let us finally turn to proving the last part of the theorem, and, at
  the same time, show that the total mass of $\bal(\sigma, 0)$ is
  precisely $\sigma(\C)$, as stated. That $\omega$ is an open set
  follows from the fact that the auxiliary function $u := U^\sigma -
  V^\sigma$ is a lower semicontinuous function with our assumptions on
  $\sigma$, and that $\omega = u^{-1}((0,\infty])$. To show that
  $\omega \subseteq \Omega$ it is enough to show that $\Delta V^\sigma
  = 0$ whenever $V^\sigma < U^\sigma$, since $\Delta V^\sigma = 0$ by
  the definition of $\bal(\sigma, 0)$ would imply $\nu = 0$ in
  $\omega$. For sake of argument let us simply assume the contrary,
  \ie there is some point where $V^\sigma < U^\sigma$ while $\Delta
  V^\sigma > 0$ holds true. Then there would exist some ball $B$
  contained in $\omega$ in which both $V^\sigma < U^\sigma$ and
  $\Delta V^\sigma > 0$ hold. In that case we could define a new
  function $\hat{V}^\sigma$ by
  \begin{gather*}
    \hat{V}^\sigma(z) := \left\{
      \begin{array}{l l}
        V^\sigma(z) & \text{on } \C \setminus B, \\
        p(z) & \text{on } B,
      \end{array}
    \right.
  \end{gather*}
  where $p$ is the Poisson integral on $B$ with boundary value
  $V^\sigma$ on $\bdry B$. This function would then be at least equal
  to $V^\sigma$ everywhere and in fact be larger than $V^\sigma$ in
  $B$ while still satisfying $\hat{V}^\sigma \leq U^\sigma$ and
  $\Delta \hat{V}^\sigma \geq 0$, hence we would have $\hat{V}^\sigma
  \in \F^\sigma_0$, contradicting the assumed maximality of $V^\sigma
  = \sup \F^\sigma_0$.

  By the definition of $V^\sigma$ it immediately follows that the
  auxiliary function $u = U^\sigma - V^\sigma$ itself can be
  considered as the smallest non-negative function satisfying
  \begin{gather}
    \label{eq:aux-func-ineq}
    -\frac{1}{2\pi} \Delta u \geq \sigma
  \end{gather}
  in $\C$. Let $R > 0$ be large enough so that $\supp \sigma \subset
  D(0, R)$; there evidently exists such an $R$ since the support of
  $\sigma$ is assumed compact. We claim that $u$ then must be harmonic
  in the set $\Theta := \C \setminus \overline{D(0,R)}$.  Indeed,
  \eqref{eq:aux-func-ineq} implies that $\Delta u \leq 0$ on $\Theta$,
  so the minimum principle for superharmonic functions implies that
  $u$ must be constantly equal to zero on the whole of $\Theta$ if it
  is zero at any point. (We in fact utilize that $u$ is assumed to be
  non-negative everywhere; a point in $\Theta$ where $u$ is zero would
  thus be a global minimum for the restriction of $u$ to any bounded
  neighborhood contained in $\Theta$, so that $u$ must be zero
  throughout that neighborhood, and we can cover $\Theta$ by such
  neighborhoods.). If, on the other hand, there is some point $z' \in
  \Theta$ where $u(z') > 0$, then $z' \in \omega =
  u^{-1}((0,\infty])$; the previous part in fact evidently implies that every
  point in $\Theta$ must be contained in $\omega$, \ie $\Theta
  \subseteq \omega$. Since we already know that $\Delta V^\sigma = 0$
  in $\omega$, and moreover also that $\Delta U^\sigma = 0$ outside
  the support of $\sigma$ (hence in $\Theta$), we obtain 
  \begin{gather*}
    \Delta u = \Delta U^\sigma - \Delta V^\sigma = 0
  \end{gather*}
  as desired.

  Having established that the difference $u$ between $U^\sigma$ and
  $V^\sigma$ must be harmonic outside the support of $\sigma$, let us
  now investigate in detail its behavior near infinity. On the one
  hand, we already know that $U^\sigma$ has the expansion
  \begin{gather*}
    U^\sigma(z) = t \log |z| + \O \left( \frac{1}{|z|} \right)
  \end{gather*}
  as $|z| \to \infty$, where $t := -\sigma(\C) > 0$. On the other
  hand, by once more applying that $U^\sigma$ is bounded from below in
  $\C$, we can find a constant $\tilde{M}$ such that the function
  $\tilde{V}$ defined by
  \begin{gather*}
    \tilde{V}(z) := t \log_+ |z| + \tilde{M}
  \end{gather*}
  is both subharmonic and satisfies $\tilde{V} \leq U^\sigma$
  everywhere (we here use the notation $\log_+ x = \max(\log x,
  0)$). It follows that $\tilde{V} \in \F^\sigma_0$, hence $\tilde{V}
  \leq V^\sigma$, and combined with $V^\sigma \leq U^\sigma$ we obtain
  \begin{gather*}
    0 \leq u(z) \leq U^\sigma(z) - t \log_+ |z| - \tilde{M};
  \end{gather*}
  as $|z| \to \infty$ this becomes
  \begin{gather*}
    0 \leq u(z) \leq \O(1).
  \end{gather*}
  We can from this infer that the minimizing function $u$ must be both
  harmonic and bounded on $\Theta$, in particular that $V^\sigma$ must
  have the behavior
  \begin{gather*}
    V^\sigma(z) = t \log |z| + \O(1)
  \end{gather*}
  as $|z| \to \infty$; as $\nu = \bal(\sigma, 0)$ is defined through
  \begin{gather*}
    \nu = \bal(\sigma,0) = - \frac{1}{2\pi} \Delta V^\sigma,
  \end{gather*}
  it follows that the logarithmic term in the expansion of $U^\nu(z)$
  for large $|z|$ must be precisely $t \log |z|$, \ie $\nu(\C) = -t =
  \sigma(\C)$. Lastly, by Weyl's Lemma
  \cite[Lemma~3.7.10]{Ransford:95} it is straightforward to see that
  the difference between $U^\nu$ and $V^\sigma$ is a harmonic function
  on $\C$ that, due to the behaviors of both $U^\nu$ and $V^\sigma$,
  must be $\O(1)$ near infinity, hence bounded, and hence constant, by
  Liouville's theorem. We thus conclude that, indeed, $V^\sigma =
  U^\nu + c_0$ for some constant $c_0$, as stated.
\end{proof}

\section{Equilibrium Measures through Partial Balayage}
\label{sec:equil-meas-thr-pb}

As a motivational example of what will follow, let us combine a few
results from weighted potential theory regarding the potential of an
equilibrium measure, and consider the weighted equilibrium problem
from the point of view as an obstacle problem. Assume $Q: E \to
(-\infty, \infty]$ is a $t$-admissible background potential on a
closed set $E \subseteq \C$ for some $t > 0$. From
Proposition~\ref{prop:char-of-eq-meas}, in particular part
\eqref{eq:req-2-for-eq-meas}, we know that $Q = F_{Q,t} -
U^{\mu_{Q,t}}$ holds q.e. on $S_{Q,t} := \supp \mu_{Q,t}$, but let us
for sake of simplicity for the moment assume that $Q$ is such that the
polar exceptional set in this equality is empty, so that we in fact
have $Q(z) = F_{Q,t} - U^{\mu_{Q,t}}(z)$ for all $z \in S_{Q,t}$ (this
is not an unnatural assumption on $Q$; all of the background
potentials treated in section~\ref{sec:examples-part-ii} all satisfy
this property). We briefly note that an empty such polar exceptional
set by Theorem~I.4.4 in \cite{SaffTotik:97} in fact implies that
$U^{\mu_{Q,t}}$ must be continuous on $\C$.

It is easily seen that if we let $\widehat{Q}$ be the restriction of $Q$
to $S_{Q,t}$ extended by positive infinity outside $S_{Q,t}$, \ie we
let $\widehat{Q}:\C \to (-\infty, \infty]$ be defined by
\begin{gather}
  \label{eq:extending-Q-by-infty}
  \widehat{Q}(z) := \left\{
    \begin{array}{ll}
      Q(z) & z \in S_{Q,t}, \\[.2cm]
      \infty & z \notin S_{Q,t},
    \end{array}
  \right.
\end{gather}
then $\widehat{Q}$ is also a $t$-admissible background potential, and,
more importantly, $\mu_{Q,t} = \mu_{\widehat{Q},t}$. We now look at
the equilibrium problem for $\widehat{Q}$ (which, as just mentioned,
is equivalent to that of $Q$) but now from an obstacle problem point
of view, and for this purpose we let $\S$ denote the set of
subharmonic functions in $\C$, and define $\S_t$ as the set of all
functions $V(z) \in \S$ that are harmonic for large $|z|$ and such
that $V(z) - t \log |z|$ is bounded from above as $|z| \to \infty$
(\cf Proposition~\ref{prop:wpt-upper-envelope}); the obstacle problem
for $\widehat{Q}$ is then to find the largest function
$V_{\widehat{Q}}$ in the class $\S_t$ that is majorized by
$\widehat{Q}$ quasi everywhere, \ie
\begin{gather}
  \label{eq:Qhat-obstacle-problem}
  V_{\widehat{Q}} = \sup\{V \in \S_t : V(z) \leq \widehat{Q}(z) \text{
    for q.e. } z \in \C\}.
\end{gather}
The reason for studying the equilibrium problem as an obstacle problem
is in part because of the fact that we want to establish a
relationship between partial balayage and the weighted equilibrium
measure, and the definition of partial balayage used in this paper is
essentially as an obstacle problem, but also in part because of the
fact that the obstacle problem formulation allows for some flexibility
in the types of obstacles allowed that does not appear in the usual
setting of weighted potential theory. In particular, we shall consider
the obstacle problem with the obstacle $\widetilde{Q} := F_{Q,t} -
U^{\mu_{Q,t}}$, now defined on the whole complex plane instead of
simply a subset of it: this will result in a solution
$V_{\widetilde{Q}}$ defined analogously to
\eqref{eq:Qhat-obstacle-problem}, and, as we will soon see, it will
turn out that the two obstacle problems for $\widehat{Q}$ and
$\widetilde{Q}$ in fact have the same solutions, \ie $V_{\widehat{Q}}
= V_{\widetilde{Q}}$. However, to specifically give an example of the
flexibility in the obstacle problem formulation, note that
$\widehat{Q}$ is a $t$-admissible background potential while
$\widetilde{Q}$ is not: as $|z|$ tends to infinity we have
$\widetilde{Q} = t \log |z| + \O(1)$, which violates admissibility
requirement (iii) in Definition~\ref{def:t-adm-bkg-pot}.

From our point of view it is the property that $\widehat{Q}$ and
$\widetilde{Q}$ as obstacles have the same solution that is
interesting. From the definitions above it is easily seen that
$\widetilde{Q}$ satisfies
\begin{align*}
  \widetilde{Q}(z) &= \widehat{Q}(z) \text{ for all } z \in
  S_{Q,t}, \\
  \widetilde{Q}(z) &\leq \widehat{Q}(z) \text{ for all } z \in \C.
\end{align*}
Combined with the previously mentioned behavior of $\widetilde{Q}$
near infinity, as well as the above mentioned property that
$U^{\mu_{Q,t}}$ is continuous on $\C$, and the (trivial) fact that
$\supp \mu_{Q,t} \subseteq S_{Q,t}$, is it relatively easily seen that
$V_{\widetilde{Q}} = V_{\widehat{Q}}$ then must hold. In fact, we can
in the above in general replace the set $S_{Q,t}$ with some compact
set $E'$, the constant $F_{Q,t}$ with some constant $c$, and the
measure $\mu_{Q,t}$ with some signed and compactly supported measure
$\sigma = \sigma_+ - \sigma_-$ as long as $\sigma(\C) = -t$,
$U^{\sigma_-}$ is continuous on $\C$, and $\supp \sigma_- \subseteq
E'$; if it then holds that the function $\widetilde{Q} := c +
U^\sigma$ satisfies
\begin{align*}
  \widetilde{Q}(z) &= \widehat{Q}(z) \text{ for q.e. } z \in E', \\
  \widetilde{Q}(z) &\leq \widehat{Q}(z) \text{ for q.e. } z \in \C,
\end{align*}
then, as we soon shall see, we will still obtain $V_{\widehat{Q}} =
V_{\widetilde{Q}}$. With this in mind, we introduce the following:

\begin{definition}
  Let $Q$ be a $t$-admissible background potential on $E \subseteq \C$
  for some $t > 0$. If $\mathbf{G} = (E', \sigma, c)$ is a triple with
  $E' \subseteq E$ a compact set, $\sigma$ a signed and compactly
  supported Radon measure with $\sigma(\C) = -t$ and $\supp \sigma_-
  \subseteq E'$, and $c \in \R$ a constant, such that the function
  $\widetilde{Q}$ defined by $\widetilde{Q}(z) := c + U^\sigma(z)$
  satisfies
  \begin{align*}
    \widetilde{Q}(z) &= Q(z) \quad \text{for q.e. } z \in E', \\
    \widetilde{Q}(z) &\leq Q(z) \quad \text{for q.e. } z \in E,
  \end{align*}
  then we say that $\mathbf{G}$ \emph{defines a $t$-extension of $Q$
    (relative to $E'$)}. Whenever it is clear from context precisely
  which $\mathbf{G}$ we are working with, we simply refer to the
  generated function $\widetilde{Q} = c + U^\sigma$ as a $t$-extension
  of $Q$.
\end{definition}

\begin{remark}
  \label{rem:exist-and-uniq-for-t-ext} From the above definition it is
  natural to ask questions of existence and uniqueness of a
  $t$-defining extension $\mathbf{G}$ for a given $t$-admissible
  background potential $Q$. As for existence, it is evident from the
  motivational example given above that $(\supp \mu_{Q,t}, -\mu_{Q,t},
  F_{Q,t})$ always defines a $t$-extension of $Q$. However, extensions
  for which $\sigma = \mu_{Q,t}$ (and $c = F_{Q,t}$ by necessity) are
  in some sense not particularly interesting: the $t$-extension
  $\widetilde{Q} = F_{Q,t} - U^{\mu_{Q,t}}$ resulting from such a
  triple is, by Proposition~\ref{prop:wpt-upper-envelope}, in fact
  \emph{itself} the solution to the obstacle problem with
  $\widetilde{Q}$ as an obstacle, \ie $V_{\widetilde{Q}} =
  \widetilde{Q}$. In order to actually have something to work with we
  are more interested in trying to find a $t$-extension
  $\widetilde{Q}$ defined by some $\mathbf{G}$ which can be
  constructed a priori and for which 
  the solution $V_{\widetilde{Q}}$ to the obstacle problem is not
  necessarily equal to $\widetilde{Q}$ everywhere. In other words, the
  interesting case for our purposes is precisely when we can find a
  $t$-extension $\widetilde{Q} = c + U^\sigma$ where $\sigma$ is
  different from $\mu_{Q,t}$. For an arbitrary $Q$ we cannot in
  general guarantee that there exists a $t$-extension for which
  $\sigma \neq \mu_{Q,t}$, and even when such an extension exists it
  may for some background potentials be the case that $\sigma$ is very
  difficult to calculate explicitly. On the other hand, in some
  settings it is relatively easy to find such a measure, and in the
  next section we will show one method that for instance works for the
  class of superharmonically perturbed Gaussian background
  potentials. However, that method will be based on us essentially
  attempting to create such a measure $\sigma$ from applying the
  Laplace operator (in the distributional sense) to $Q$. Applying this
  on a more general background potential need not always be fruitful,
  since $\Delta Q$ need not always be a measure; for instance, the
  background potential $Q$ could have jump discontinuities, so that
  $\Delta Q$ becomes a distribution of non-zero order.
\end{remark}

\begin{remark}
  \label{rem:E-subset-complex-plane}
  Just like we extended the background potential by infinity in our
  motivational example in \eqref{eq:extending-Q-by-infty}, we can
  always perform the same sort of extension on any general background
  potential $Q$. Thus, we may just as well think of the set $E
  \subseteq \C$ on which $Q$ is defined to be the entire complex
  plane, by simply letting $Q(z) := \infty$ for any $z \notin E$.
\end{remark}

\noindent Having defined the notion of a $t$-extension, we are now
ready to state a complementarity theorem between weighted equilibrium
measures and partial balayage.

\begin{theorem}
  \label{thm:wpt-pb-rel}
  Let $Q$ be a $t$-admissible background potential on $E \subseteq
  \C$, and assume $\mathbf{G} = (E', \sigma, c)$ defines a
  $t$-extension $\widetilde{Q} = c + U^\sigma$ of $Q$ relative to
  $E'$. If $U^{\sigma_-}$ is continuous on $\C$, then
  \begin{gather}
    \label{eq:main-eq-meas-pb-relation}
    \mu_{Q,t} + \bal(\sigma, 0) = 0.
  \end{gather}
\end{theorem}

\begin{proof}
  As mentioned in Remark~\ref{rem:E-subset-complex-plane}, we will,
  without loss of generality, consider $Q$ to be defined on the entire
  complex plane, \ie $E = \C$. Under the given assumptions on
  $\sigma$, it follows from
  Theorem~\ref{thm:pb-properties-with-assumptions} that $\bal(\sigma,
  0)$ exists. Let, as before, $\S$ be the set of subharmonic functions
  in $\C$, and define $\S_t$ to be the set of all functions $V \in \S$
  that are harmonic in a neighborhood of infinity and such that $V(z)
  - t \log |z|$ is bounded from above as $|z| \to \infty$, \ie $V(z)
  \leq t \log |z| + \O(1)$ for large $|z|$. We now consider the two
  obstacle problems
  \begin{gather*}
    V_{Q} := \sup \{ V \in \S_t : V(z) \leq Q(z) \text{ for q.e. } z
    \in E = \C\}
  \end{gather*}
  and
  \begin{align*}
    V_{\widetilde{Q}} :=& \sup \{ V \in \S_t : V(z) \leq \widetilde{Q}(z)
    \text{ for q.e. } z \in \C\} \\
    =& \sup \{ V \in \S : V(z) \leq c + U^{\sigma}(z) \text{ for all }
    z \in \C\}.
  \end{align*}
  The last equality above is motivated in part by that quasi
  everywhere implies everywhere for this particular obstacle: the
  inequality is equivalent with $h_1(z) + h_2(z) \geq 0$ q.e.\ for
  $h_1(z) = U^{\sigma_+}(z) - V(z)$, $h_2(z) = c - U^{\sigma_-}(z)$,
  and under the given assumptions we get that $h_1$ is superharmonic
  and $h_2$ is continuous everywhere, so a standard potential
  theoretic argument using mollifiers can be used. With this property
  established, it is evident that the solution $V_{\widetilde{Q}}$ to
  the second obstacle problem in fact must be precisely
  $V_{\widetilde{Q}} = V^\sigma + c$, where $V^\sigma \equiv
  V^\sigma_0 = \sup \F^\sigma_0$ is as in the definition of
  $\bal(\sigma, 0)$, since we from
  Theorem~\ref{thm:pb-properties-with-assumptions} can conclude that
  we must have $V^\sigma + c \in \S_t$. By the same theorem we obtain
  $V^\sigma = c_0 + U^{\bal(\sigma, 0)}$ for some $c_0$, and hence, if
  we for sake of simplicity let $\nu$ be the positive measure $\mu :=
  - \bal(\sigma, 0)$, we can conclude that
  \begin{gather}
    \label{eq:V_tilde-Q-as-potential}
    V_{\widetilde{Q}} = \widetilde{c} - U^\mu
  \end{gather}
  for some constant $\widetilde{c} = c + c_0$.  As for the solution to
  the first obstacle problem, we simply apply
  Proposition~\ref{prop:wpt-upper-envelope} to immediately obtain that
  we must have
  \begin{gather}
    \label{eq:V_Q-as-potential}
    V_Q = F_{Q,t} - U^{\mu_{Q,t}}.
  \end{gather}

  We now claim that $V_Q = V_{\widetilde{Q}}$. In other words, as
  earlier mentioned, from an obstacle problem point of view one could
  replace the obstacle $Q$, used in determining the potential of the
  weighted equilibrium measure, with the $t$-extension $\widetilde{Q}$
  and still get the same solution. Since the solution to the obstacle
  problem for the obstacle $\widetilde{Q}$ evidently can be expressed
  in terms of the partial balayage measure $\bal(\sigma, 0)$, this
  will establish the proposed complementarity relationship
  \eqref{eq:main-eq-meas-pb-relation}.

  As for proving the equality between $V_Q$ and $V_{\widetilde{Q}}$,
  we show that either one must be less than or equal to the
  other. Since we by definition of our $t$-extension $\widetilde{Q}$
  have $\widetilde{Q} \leq Q$ quasi everywhere, it is clear that
  $V_{\widetilde{Q}} \leq V_Q$ holds, and so the main difficulty is
  showing the we also have the reverse inequality, \ie that $V_Q \leq
  V_{\widetilde{Q}}$. Utilizing \eqref{eq:V_tilde-Q-as-potential} and
  \eqref{eq:V_Q-as-potential} it is evident that showing $V_Q \leq
  V_{\widetilde{Q}}$ is equivalent with showing that
  \begin{gather*}
    U^\mu \leq U^{\mu_{Q,t}} + (\widetilde{c} - F_{Q,t})
  \end{gather*}
  holds everywhere. Using the Principle of Domination,
  Proposition~\ref{prop:princ-of-domination}, we get that it in fact
  is enough to show that $V_Q \leq V_{\widetilde{Q}}$ holds on the
  support of $\mu = -\bal(\sigma, 0)$. However, this is easy: from
  Theorem~\ref{thm:pb-properties-with-assumptions} it follows that
  $V^\sigma = U^\sigma$ on $\supp \bal(\sigma, 0)$, hence
  \begin{gather*}
    V_{\widetilde{Q}} = c + V^\sigma = c + U^\sigma = \widetilde{Q}
  \end{gather*}
  holds there. But from the same theorem and one of our starting
  assumptions we moreover also know that $\supp \bal(\sigma, 0)
  \subseteq \supp \sigma_- \subseteq E'$. Since the $t$-extension
  $\widetilde{Q}$ is defined to satisfy $\widetilde{Q} = Q$ q.e. on
  the set $E'$, it now follows that $V_{\widetilde{Q}} = Q$ holds on
  $\supp \bal(\sigma, 0)$. From this it is evident that $V_Q \leq
  V_{\widetilde{Q}}$ holds on $\supp \bal(\sigma, 0)$, and we can
  therefore finally conclude that we indeed have $V_Q \leq
  V_{\widetilde{Q}}$ everywhere, and hence
  \begin{gather*}
    F_{Q,t} - U^{\mu_{Q,t}} = V_Q = V_{\widetilde{Q}} = \widetilde{c} +
    U^{\bal(\sigma, 0)}
  \end{gather*}
  holds everywhere. Applying the Laplace operator on the leftmost and
  rightmost sides of this last equality to recover the measures yields
  \eqref{eq:main-eq-meas-pb-relation}.
\end{proof}

There are a few corollaries to Theorem~\ref{thm:wpt-pb-rel} that we
should mention. One of the main difficulties in determining the
weighted equilibrium measure $\mu_{Q,t}$ is to locate its support, and
for this reason the following result may be interesting:

\begin{corollary}
  Let $Q$ and $\mathbf{G} = (E', \sigma, c)$ be as in
  Theorem~\ref{thm:wpt-pb-rel}, with $\supp \sigma_- \subseteq E'$ and
  $U^{\sigma_-}$ assumed continuous on $\C$. Then $\supp \mu_{Q,t}
  \subseteq E'$.
\end{corollary}

\begin{proof}
  This follows immediately from applying the theorem, along with
  \begin{gather*}
    \supp \mu_{Q,t} = \supp \bal(\sigma, 0) \subseteq \supp \sigma_-
    \subseteq E'.
  \end{gather*}
\end{proof}

\noindent Also, if the set $E$ on which $Q$ is assumed to be defined
is already a compact set, then Theorem~\ref{thm:wpt-pb-rel} can be
stated in a slightly simplified way:

\begin{corollary}
  \label{cor:wpt-pb-rel-compact}
  Let $E \subset \C$ be a compact set, and let $Q$ be an admissible
  background potential on $E$. Assume there exists a signed Radon
  measure $\sigma = \sigma_+ - \sigma_-$ with compact support such
  that $Q = \restr{U^{\sigma}}{E} + c$ for some constant $c \in \R$
  and such that $\sigma$ satisfies $t := - \sigma(\C) > 0$, $\supp
  \sigma_- \subseteq E$ and $U^{\sigma_-}$ is continuous on $\C$. Then
  $\mu_{Q,t} + \bal(\sigma, 0) = 0$.
\end{corollary}

\begin{proof}
  Simply apply Theorem~\ref{thm:wpt-pb-rel} using $\mathbf{G} = (E,
  \sigma, c)$.
\end{proof}

\section{Examples: Part II}
\label{sec:examples-part-ii}

Let us now once more turn our focus to the sample potential 
\begin{gather*}
  Q(z) = \alpha |z|^2 + \beta \log \frac{1}{|z-a|},
\end{gather*}
with $\alpha, \beta > 0$ and $a \in \C$ arbitrary, and attempt to find
the equilibrium measure $\mu_{Q,t}$ using the complementarity
relationship to partial balayage given in
Theorem~\ref{thm:wpt-pb-rel}, \ie we essentially want to verify the
results in section~\ref{sec:examples-part-i} using partial
balayage. In fact, the precise technique we will use will work for any
so-called Gaussian background potential with a superharmonic
perturbation (as defined in \cite{BaloghHarnad:09}): we will in this
section from hereon hence assume that our background potential $Q$ has
the form
\begin{gather}
  \label{eq:ex-pert-gauss-bkgpot}
  Q(z) = \alpha |z|^2 + U^\nu(z),
\end{gather}
where $\alpha > 0$ and $\nu$ is a finite positive Borel measure with
compact support (note that our sample potential corresponds to $\nu =
\beta \delta_a$ with $\beta > 0$ and $a \in \C$). Throughout this
section we will moreover assume that
$t > 0$ is an arbitrary but fixed positive real number.\\

\noindent The method we will use will of course be that of finding a
suitable $t$-extension $\widetilde{Q}$ of the background potential
$Q$, and apply the theorem given in the previous section. Finding such
a $t$-extension is not a trivial task in general, but for $Q$ of the
form \eqref{eq:ex-pert-gauss-bkgpot} it turns out that the following
is suitable: let $\rho > 0$ be a constant, soon to be determined, let
$\widetilde{Q} = c + U^\sigma$ denote our desired $t$-extension
relative to some compact set $E'$ (\ie $\mathbf{G} := (E', \sigma,
c)$), and in particular take $E' := \overline{\disk{0}{\rho}}$; we
thus need to find a signed Radon measure $\sigma$ of compact support
and a constant $c$ such that the function $\widetilde{Q} = c +
U^\sigma$ satisfies $\widetilde{Q}(z) = Q(z)$ if $|z| \leq \rho$ and
$\widetilde{Q}(z) \leq Q(z)$ if $|z| > \rho$ (see
Figure~\ref{fig:ex-obstacle-problem} for an illustration of the case
where $\nu = \beta \delta_a$). The way we will do this is to simply
try to construct $\widetilde{Q}$ so that it gets the desired behavior
on $\C$, and then simply let $\sigma$ be defined by
\begin{gather*}
  \sigma := - \frac{1}{2\pi} \Delta \widetilde{Q};
\end{gather*}
it will soon become clear that there then exists some constant $c$
such that we indeed have $\widetilde{Q} = c + U^\sigma$.
\begin{center}
  \begin{figure} \centering
    \begin{minipage}[t]{1.0\textwidth} \centering
      \includegraphics[scale=.555]{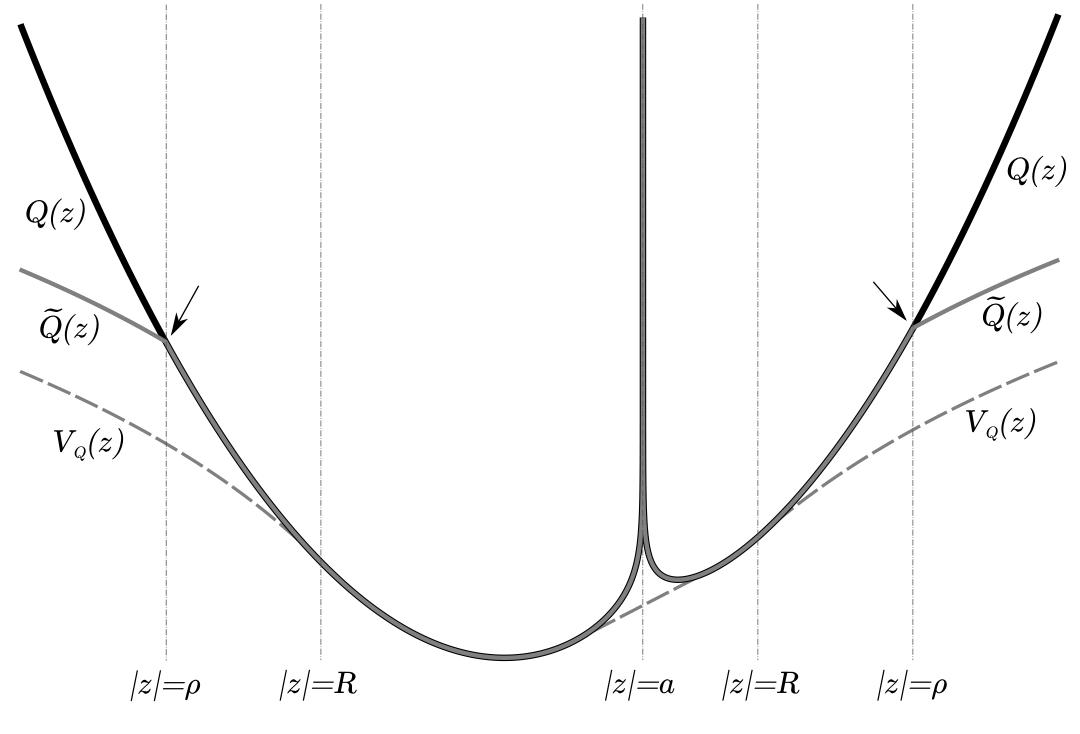}
    \end{minipage}
    \caption{A plot through the line containing the origin and the
      point $a$ of $Q(z) = \alpha |z|^2 + \beta \log \frac{1}{|z-a|}$,
      the $t$-extension $\widetilde{Q}$, equal to $Q$ on
      $E'=\overline{\disk{0}{\rho}}$ and $\widetilde{Q}(z) \sim t \log
      |z| + \O(1)$ near infinity, as well as the solution $V_Q$
      ($=V_{\widetilde{Q}}$) to the obstacle problem with obstacle
      $Q$; as can be seen in the figure $\rho$ is assumed to satisfy
      $\rho > R = \sqrt{\frac{t + \beta}{2\alpha}}$. Note the
      discontinuity arising in the radial derivative of
      $\widetilde{Q}$ illustrated by the arrows at $|z|=\rho$; the
      obtained $t$-extension $\widetilde{Q}$ is $C^1$ at $|z|=\rho$ if
      and only if $\rho=R$.}
    \label{fig:ex-obstacle-problem}
  \end{figure}
\end{center}
Since our background potential $Q$ is the sum of a Gaussian term
$\alpha |z|^2$ with another term that is already a potential, let us
attempt to find our $t$-extension $\widetilde{Q}$ by simply modifying
the part of $Q$ that is not already the potential of some measure, \ie
we will focus on the term $\alpha |z|^2$. For the relation
$\widetilde{Q}=c+U^\sigma$ to hold we are going to require that
\begin{align}
  \widetilde{Q}(z) &= \sigma(\C) \log \frac{1}{|z|} + \O(1) 
= (t + \nu(\C)) \log |z| + \nu(\C) \log \frac{1}{|z|} +
  \O(1) \label{eq:Q-tilde-form-near-infty} 
\end{align}
holds near infinity. As the second term on the rightmost side
corresponds to the behavior of $U^\nu$ near infinity, let us simply
base our ansatz for $\widetilde{Q}$ directly on
\eqref{eq:Q-tilde-form-near-infty}, and therefore see if we can find a
constant ${c} \in \R$ such that
\begin{gather}
  \label{eq:Q-tilde-def-outside}
  \widetilde{Q}(z) = (t + \nu(\C)) \log |z| + U^\nu(z) + {c}
\end{gather}
holds for all $|z| > \rho$, with $\widetilde{Q}$ still satisfying the
required properties discussed above. The value of ${c}$ is easily
determined: we obviously need that the resulting function
$\widetilde{Q}$ is a.e.\ continuous at the points on the set $\{z : |z|
= \rho\}$, and to achieve this while demanding that $\widetilde{Q}(z)
= Q(z)$ for $|z| < \rho$ and $\widetilde{Q}$ given by
\eqref{eq:Q-tilde-def-outside} for $|z| > \rho$, it is clear that we
need
\begin{gather*}
  {c} = \alpha \rho^2 - (t + \nu(\C)) \log \rho.
\end{gather*}
As for the requirement that $\widetilde{Q} \leq Q$ must hold
everywhere, we see that this is clearly the case if we have
\begin{gather}
  \label{eq:sample-pot-req-ineq}
  (t + \nu(\C)) \log |z| + {c} \leq \alpha |z|^2
\end{gather}
for all $|z| \geq \rho$. For sake of obtaining this property, let $f :
[\rho, \infty) \to \R$ be defined by
\begin{gather*}
  f(x) := \alpha x^2 - (t + \nu(\C)) \log x - {c};
\end{gather*}
the inequality \eqref{eq:sample-pot-req-ineq} is evidently equivalent
with $f$ being a non-negative function on $[\rho, \infty)$. On the one
hand it is clear from how we defined ${c}$ that $f(\rho) =
0$. Moreover, we have
\begin{gather*}
  f'(x) = 2 \alpha x - \frac{t + \nu(\C)}{x},
\end{gather*}
and so non-negativity of $f$ follows if we have $f'(x) \geq 0$ for all
$x \geq \rho$. It is evident that $f'$ is monotonically increasing, so
we only need to demand that $f'(\rho) \geq 0$; hence we require that
\begin{gather}
  \label{eq:sample-class-defining-radius-for-positivity}
  2\alpha \rho - \frac{t + \nu(\C)}{\rho} \geq 0 \iff \rho \geq
  \sqrt{\frac{t + \nu(\C)}{2\alpha}} =: R.
\end{gather}
As earlier suggested we now define $\sigma$ through the relation $-
\Delta \widetilde{Q} = 2 \pi \sigma$. An easy calculation, taking into
account the possible measure arising on the set $\{z : |z| = \rho\}$
due to the first derivative of $\widetilde{Q}$ in the radial direction
possibly being discontinuous there, yields that
\begin{gather*}
  \sigma = - \frac{1}{2\pi} \Delta \widetilde{Q} = - \frac{2\alpha}{\pi}
  \restrx{m}{\disk{0}{\rho}} + \nu + \frac{1}{2\pi} \left( 2\alpha
    \rho - \frac{t + \nu(\C)}{\rho} \right) s,
\end{gather*}
where $s$ is the arc length measure on $\{z : |z| = \rho\}$ of total
mass $2 \pi \rho$. Using the standard results from classical potential
theory that the potential $U^{\disk{0}{\rho}}$ of the Lebesgue measure
restricted to the disk $\disk{0}{\rho}$ is given by
\begin{gather}
  \label{eq:log-pot-of-disk-at-origin}
  U^{\disk{0}{\rho}}(z) = \left\{
    \begin{array}{l l}
      \displaystyle - \frac{\pi}{2} |z|^2 + \frac{\pi \rho^2}{2} \left(
        \log \frac{1}{\rho^2} + 1
      \right) & |z| \leq \rho, \\[.2cm]
      \displaystyle \pi \rho^2 \log \frac{1}{|z|}& |z| > \rho,
    \end{array}
  \right.
\end{gather}
and that the potential $U^s$ of the arc length measure $s$ on $\{z :
|z| = \rho\}$ is given by
\begin{gather*}
  U^s(z) = \left\{
    \begin{array}{l l}
      \displaystyle 2 \pi \rho \log \frac{1}{\rho} & |z| \leq \rho, \\[.2cm]
      \displaystyle 2 \pi \rho \log \frac{1}{|z|} & |z| > \rho,
    \end{array}
  \right.
\end{gather*}
one easily sees that for $|z| \leq \rho$ we obtain
\begin{align*}
  U^\sigma(z) &= - \frac{2\alpha}{\pi} U^{\disk{0}{\rho}} + U^\nu(z) +
  \frac{1}{2\pi} \left( 2\alpha \rho - \frac{t + \nu(\C)}{\rho}
  \right) U^s(z) \\
  &= \alpha |z|^2 + \alpha \rho^2 \log \rho^2 - \alpha \rho^2 +
  U^\nu(z) - 2 \alpha \rho^2 \log \rho + (t + \nu(\C)) \log \rho \\
  &= \alpha |z|^2 + U^\nu(z) - (\alpha \rho^2 - (t + \nu(\C)) \log
  \rho) \\ &= Q(z) - {c}.
\end{align*}
For $|z| > \rho$ we in a similar way instead get
\begin{align*}
  U^\sigma(z) &= 2 \alpha \rho^2 \log |z| + U^\nu(z) - 2 \alpha \rho^2
  \log |z| + (t + \nu(\C)) \log |z| \\
  &= (t + \nu(\C)) \log |z| + U^\nu(z).
\end{align*}
We can summarize these results as
\begin{gather*}
  U^\sigma(z) + {c} = \left\{
    \begin{array}{l l}
      Q(z) & |z| \leq \rho, \\[.2cm]
      (t + \nu(\C)) \log |z| + U^\nu(z) + {c} & |z| > \rho,
    \end{array}
  \right.
\end{gather*}
which shows that we indeed obtain $\widetilde{Q} = c + U^\sigma$
everywhere for our choice of $\widetilde{Q}$. Since we have already
seen that this $\widetilde{Q}$ satisfies $\widetilde{Q} \leq Q$
everywhere, it follows that $\mathbf{G} := (E', \sigma, c)$ defines a
$t$-extension of $Q$ if we have $\sigma(\C) = -t$; this property of
course follows immediately from that we originally \emph{defined}
$\sigma$ so that the potential of it should behave like $t \log |z| +
\O(1)$ near infinity, but we can just as well also calculate it
explicitly:
\begin{align*}
  \sigma(\C) &= -\frac{2\alpha}{\pi} \cdot \pi \rho^2 + \nu(\C) +
  \frac{1}{2\pi} \left( 2\alpha \rho - \frac{t + \nu(\C)}{\rho}
  \right) \cdot 2 \pi \rho \\
  &= - 2 \alpha \rho^2 + \nu(\C) + 2 \alpha \rho^2 - (t + \nu(\C)) =
  -t. 
\end{align*}
Finally, from \eqref{eq:sample-class-defining-radius-for-positivity}
we see that we have $\sigma = \sigma_+ - \sigma_-$ with
\begin{align*}
  \sigma_+ &= \nu + \frac{1}{2\pi} \left(
    2 \alpha \rho - \frac{t + \nu(\C)}{\rho}
  \right) s, \\
  \sigma_- &= \frac{2\alpha}{\pi} \restrx{m}{\disk{0}{\rho}}.
\end{align*}
Since it then is evident that $\supp \sigma_- =
\overline{\disk{0}{\rho}} = E'$ and that the potential $U^{\sigma_-}$
is continuous on the complex plane
(cf. \eqref{eq:log-pot-of-disk-at-origin}), it finally follows, by an
application of Theorem~\ref{thm:wpt-pb-rel}, that we indeed have
$\mu_{Q,t} + \bal(\sigma, 0) = 0$. The support of $\mu_{Q,t}$ must
especially hence be contained in $\disk{0}{\rho}$ for every choice of
$\rho$ satisfying
\begin{gather*}
  \rho \geq \sqrt{\frac{t + \nu(\C)}{2\alpha}},
\end{gather*}
so by taking the smallest such $\rho$ we can conclude the following
proposition, which in a similar form was conjectured in
\cite{BaloghHarnad:09} but not proven there:
\begin{proposition}
  Let $Q(z) = \alpha|z|^2 + U^\nu(z)$ be a Gaussian background
  potential with a superharmonic perturbation, \ie we assume that
  $\alpha > 0$ and $\nu$ is a finite positive Borel measure with
  compact support. Then
  \begin{gather*}
    \supp \mu_{Q,t} \subseteq \overline{\disk{0}{\sqrt{\frac{t +
            \nu(\C)}{2\alpha}}}}. 
  \end{gather*}
\end{proposition}
\noindent As a final remark we note that
\begin{gather*}
  \rho = \sqrt{\frac{t + \nu(\C)}{2\alpha}} \implies 2\alpha \rho -
  \frac{t + \nu(\C)}{\rho} = 0,
\end{gather*}
\ie the smallest possible value for $\rho$ that still makes
$\mathbf{G} = (E', \sigma, c)$ a $t$-extension is precisely the one
that makes the arc measure supported on $\{z : |z| = \rho\}$ in
$\sigma$ vanish; this is precisely the radius required to make the
radial derivative of the obtained $t$-extension $\widetilde{Q}$
continuous in a neighborhood of $\{|z|=\rho\}$
(cf. Figure~\ref{fig:ex-obstacle-problem}).

\bibliographystyle{abbrv}
\bibliography{refs}

\begin{thebibliography}{10}

\bibitem{BaloghHarnad:09}
F.~Balogh and J.~Harnad.
\newblock Superharmonic perturbations of a {G}aussian measure, equilibrium
  measures and orthogonal polynomials.
\newblock {\em Complex Anal. Oper. Theory}, 3(2):333--360, 2009.

\bibitem{Frostman:35}
O.~Frostman.
\newblock {\em Potentiel d'{\'e}quilibre et capacit{\'e} des ensembles}.
\newblock PhD thesis, Lund University, 1935.

\bibitem{GardinerSjodin:08}
S.~J. Gardiner and T.~Sj{\"o}din.
\newblock Convexity and the exterior inverse problem of potential theory.
\newblock {\em Proc. Amer. Math. Soc.}, 136(5):1699--1703, 2008.

\bibitem{GardinerSjodin:09}
S.~J. Gardiner and T.~Sj{\"o}din.
\newblock Partial balayage and the exterior inverse problem of potential
  theory.
\newblock In D.~Bakry et~al., editors, {\em Potential theory and stochastics in
  Albac}, pages 111--123. Theta, Bucharest, 2009.

\bibitem{Gustafsson:04}
B.~Gustafsson.
\newblock Lectures on balayage.
\newblock {\em Clifford algebras and potential theory, Univ. Joensuu Dept.
  Math. Rep. Ser.}, 7:17--63, 2004.

\bibitem{GustafssonSakai:94}
B.~Gustafsson and M.~Sakai.
\newblock Properties of some balayage operators, with applications to
  quadrature domains and moving boundary problems.
\newblock {\em Nonlinear Anal.}, 22(10):1221--1245, 1994.

\bibitem{HedenmalmMakarov:13}
H.~Hedenmalm and N.~Makarov.
\newblock Coulomb gas ensembles and {L}aplacian growth.
\newblock {\em Proc. Lond. Math. Soc. (3)}, 106(4):859--907, 2013.

\bibitem{Ransford:95}
T.~Ransford.
\newblock {\em Potential theory in the complex plane}, volume~28 of {\em London
  Math. Soc. Stud. Texts}.
\newblock Cambridge University Press, 1995.

\bibitem{SaffTotik:97}
E.~B. Saff and V.~Totik.
\newblock {\em Logarithmic potentials with external fields}, volume 316 of {\em
  Grundlehren der Mathematischen Wissenschaften [Fundamental Principles of
  Mathematical Sciences]}.
\newblock Springer-Verlag, 1997.
\newblock Appendix B by Thomas Bloom.

\bibitem{Sjodin:07}
T.~Sj{\"o}din.
\newblock On the structure of partial balayage.
\newblock {\em Nonlinear Anal.}, 67(1):94--102, 2007.

\bibitem{Teodorescu:06}
R.~Teodorescu.
\newblock Generic critical points of normal matrix ensembles.
\newblock {\em J. Phys. A}, 39(28):8921--8932, 2006.

\bibitem{TeodorescuEtAl:05}
R.~Teodorescu, E.~Bettelheim, O.~Agam, A.~Zabrodin, and P.~Wiegmann.
\newblock Normal random matrix ensemble as a growth problem.
\newblock {\em Nuclear Phys. B}, 704(3):407--444, 2005.

\bibitem{WiegmannZabrodin:03}
P.~Wiegmann and A.~Zabrodin.
\newblock Large scale correlations in normal non-{H}ermitian matrix ensembles.
\newblock {\em J. Phys. A}, 36(12):3411--3424, 2003.

\end{thebibliography}

\vspace{.2cm}
{\small \noindent 
Joakim Roos\\
Department of Mathematics\\ 
KTH\\
SE-100 44 Stockholm\\
Sweden\\
e-mail: {\ttfamily joakimrs@math.kth.se}}

\end{document}